\documentclass[10pt,reqno]{amsart}
\usepackage{amsmath} 
\usepackage{amssymb}
\usepackage{dsfont}
\usepackage[dvips,draft,final]{graphics}
\usepackage{amssymb,amsmath}
\usepackage[T1]{fontenc}
\usepackage{fancyhdr}
\usepackage{url}

\usepackage{color}
\usepackage{graphicx}

\def\squarebox#1{\hbox to #1{\hfill\vbox to #1{\vfill}}}

\theoremstyle{plain}

\newtheorem{Thm}{Theorem}

\newtheorem{lem}{Lemma}

\newcommand{\bel}{\begin{equation} \label}
\newcommand{\ee}{\end{equation}}
\newcommand{\re}{\mathfrak R}
\newcommand{\im}{\mathfrak I}

\newcommand{\R}{\mathbb{R}}

\newcommand{\C}{{\mathbb C}}

\def\epsilon{\varepsilon}
\def\phi {\varphi}

\newtheorem{rem}{Remark}
\newtheorem{prop}{Proposition}
\newtheorem{defi}{Definition}

\providecommand{\abs}[1]{\left\lvert#1\right\rvert}
\providecommand{\norm}[1]{\left\lVert#1\right\rVert}
\numberwithin{equation}{section}

\renewcommand{\leq}{\leqslant}
\renewcommand{\geq}{\geqslant}
\providecommand{\abs}[1]{\left\lvert#1\right\rvert}
\providecommand{\norm}[1]{\left\lVert#1\right\rVert}

\def\beq{\begin{equation}}
\def\eeq{\end{equation}}
\newcommand{\bea}{\begin{eqnarray}}
\newcommand{\eea}{\end{eqnarray}}
\newcommand{\beas}{\begin{eqnarray*}}
\newcommand{\eeas}{\end{eqnarray*}}

{

\begin{document}

\title[Determination of a time-dependent coefficient for  wave equations]{Stability in the determination of a time-dependent coefficient for  wave equations from partial   data}

\author[Yavar Kian]{Yavar Kian}
\maketitle
\begin{center} \footnotesize{ CPT, UMR CNRS 7332,\\ Aix Marseille Universit\'e,\\ 13288 Marseille, France,\\ and Universit\'e de Toulon,\\ 83957 La Garde, France\\ yavar.kian@univ-amu.fr}\end{center}

\begin{abstract}
We consider the stability in the inverse problem consisting of the determination of  a time-dependent coefficient of order zero $q$, appearing in a Dirichlet initial-boundary value problem for  a wave equation $\partial_t^2u-\Delta u+q(t,x)u=0$ in $Q=(0,T)\times\Omega$ with $\Omega$  a  $C^2$ bounded domain of $\R^n$, $n\geq2$, from  partial observations on $\partial Q$. The observation is given by a boundary operator associated to the wave equation. Using suitable complex geometric optics solutions and  Carleman estimates, we prove a stability estimate in the determination of $q$ from the  boundary operator. \\


\end{abstract}

\section{Introduction}
\label{sec-intro}
\setcounter{equation}{0}
\subsection{Statement of the problem }
In the present paper we consider   a $\mathcal C^2$ bounded domain $\Omega$ of $\R^n$, $n\geq2$.
We set  $\Sigma=(0,T)\times\partial\Omega$ and $Q=(0,T)\times\Omega$ with $0<T<\infty$.  We introduce the  following initial-boundary
value problem (IBVP in short) for the wave equation
\begin{equation}\label{eq1}\left\{\begin{array}{ll}\partial_t^2u-\Delta u+q(t,x)u=0,\quad &\textrm{in}\ Q,\\  u(0,\cdot)=v_0,\quad \partial_tu(0,\cdot)=v_1,\quad &\textrm{in}\ \Omega,\\ u=g,\quad &\textrm{on}\ \Sigma,\end{array}\right.\end{equation}
where the  potential $q\in L^\infty(Q)$ is assumed to be real valued.  We study the inverse problem of determining $q$ from observations  of the solutions of \eqref{eq1} on $\partial Q$. Let us introduce the hyperbolic Dirichlet-Neumann (DN in short) map associated to \eqref{eq1}, with $v_0=v_1=0$, defined by
$\Lambda_q:g\mapsto\partial_\nu u$ with $u$ the solution of  \eqref{eq1}, with $v_0=v_1=0$, and $\nu$  the outward unit normal vector to $\Omega$.
It is well known that the DN map determines uniquely a time-independent potential $q$ (e.g. \cite{RS1}).   In contrast to time-independent potentials, due to domain of dependence arguments, there is no hope to recover  general time dependent potential $q$ (see \cite[Subsection 1.1]{Ki2}) from the DN map $\Lambda_q$ on the set
\[D=\{(t,x)\in Q:\  0<t<\textrm{Diam}(\Omega)/2,\  \textrm{dist}(x,\partial\Omega)< t\}.\]
Facing this obstruction to uniqueness it seems that the minimal data   that allows to recover  globally general time-dependent potential $q$ so far (at finite time) is given by \cite[Theorem 1]{Ki2} where uniqueness is stated. The main goal of the present paper is to prove stability in the recovery of some general time-dependent potentials $q$  from similar data. 

Practically, our inverse problem is to determine physical properties such as the time evolving density of an inhomogeneous medium by probing it with disturbances generated on the boundary and at initial time. The data is the response of the medium to these disturbances, measured on the boundary and at the end of the experiment, and the purpose is to recover the function which measures the property of the medium. Note that time-dependent potential can also be considered for models where the evolution in time of the perturbation can not be avoided.

We also remark that, according to \cite{I2}, the determination of time-dependent potentials can  be an important tool for the  determination of a semilinear term appearing in a semilinear hyperbolic equation from boundary measurements.
\subsection{Existing papers }

In recent years the problem of recovering  coefficients for hyperbolic equations  from boundary measurements  has attracted many attention.  Many authors have considered this problem with an observation given by the DN map $\Lambda_q$.  In \cite{RS1}, the authors proved that the DN map determines uniquely the time-independent  potential in a wave equation and in   \cite{I1} Isakov considered the determination of a coefficient of order zero and a damping coefficient. These results are concerned with measurements on  the whole boundary. The uniqueness by local DN map has been considered in \cite{E1}. The stability estimate in the case where the DN map is considered on the whole lateral boundary was treated by Stefanov and Uhlmann \cite{SU,SU2}. The uniqueness and H\"older stability estimate in a subdomain were established by Isakov and Sun \cite{IS} and, assuming that the coefficients are known in a neighborhood of the boundary, Bellassoued,  Choulli and  Yamamoto \cite{BCY} proved a log-type stability estimate in the case where the Neumann data are observed in an arbitrary subdomain of the boundary. We mention also  \cite{Mo}, where the stability issue have been considered for large class of coefficients. In a recent work \cite{Ki} extended the results of \cite{RS1} to determine a large class of time-independent coefficients of order zero in an unbounded cylindrical domain. It has been proved that only measurements on a bounded subset are required for the determination of some class of coefficients including periodic coefficients and compactly supported coefficients.

Let us also mention that the method using Carleman inequalities was first considered  by Bukhgeim and Klibanov \cite{BK}. For the application of Carleman estimates to the problem of recovering time-independent coefficients for hyperbolic equations we refer to \cite{B}, \cite{IY} and \cite{K}. 

All the above mentioned results are concerned only with time-independent coefficients. Several authors considered the problem of determining time-dependent coefficients for hyperbolic equations. In \cite{St}, Stefanov proved unique determination of a time-dependent potential for the wave equation  from the knowledge of scattering data. The result of \cite{St} is equivalent to the consideration of the problem with boundary measurements.   In \cite{RS}, Ramm and  Sj\"ostrand considered the problem of determining the time-dependent coefficient $q$ from DN map $\Lambda_q$ associated to \eqref{eq1}. For this purpose, they  considered the   problem on the infinite time-space cylindrical domain $\R_t\times\Omega$  instead of $Q$ ($t\in\R$ instead of $0<t<T<\infty$) and their DN map was associated to solutions vanishing for large negative time. Then, under suitable additional assumptions, \cite{RS} proved a result of uniqueness. The result of \cite{RS} has been extended to more general coefficient by \cite{S} where stability estimate is also stated  for compactly supported coefficients provided $T$ is sufficiently large.  In  \cite{RR}, Rakesh and  Ramm considered the same problem at finite time on $Q$, with $T>\textrm{Diam} (\Omega)$, and they proved a  uniqueness result for the determination of $q$ restricted to the subset $S$ of $Q$, 
made of lines  with angle $45^\circ$    with the $t$-axis and which meet the planes $t = 0$ and $t = T$ outside
$\overline{Q}$, from  the DN map $\Lambda_q$. In Theorem 4.2 of \cite{I}, Isakov established a result of uniqueness for a time-dependent potential on the whole domain $Q$ from observations of the solution on $\partial Q$. Applying  a result of unique continuation borrowed from \cite{T}, Eskin \cite{E2} proved that the DN map uniquely determines time-dependent coefficients that are analytic with respect to the time variable $t$. In some recent work,  \cite{W} proved stability in the recovery of  X-ray transforms of time-dependent potentials on a Riemannian manifold. We also mention that \cite{A}, proved log-type stability in the recovery of time-dependent potentials from the data considered by \cite{RR} and \cite{I}. Finally in \cite{Ki2},  the author proved determination of general time dependent potential from, roughly speaking, half of the data considered by \cite{I}.
	
We also mention that \cite{Ch}, \cite{CK}, \cite{CKS} and \cite{GK} consider the problem of determining a time-dependent coefficient for parabolic and Schr\"odinger equations and  derive stability estimate for these problems.

\subsection{Main result}
In order to state our main result, we  first introduce some intermediate tools and notations. For all $\omega\in\mathbb S^{n-1}=\{y\in\R^n:\ \abs{y}=1\}$ we introduce the $\omega$-illuminated face
\[\partial\Omega_{-,\omega}=\{x\in\partial\Omega:\ \nu(x)\cdot\omega\leq0\}\]
and the $\omega$-shadowed face 
\[\partial\Omega_{+,\omega}=\{x\in\partial\Omega:\ \nu(x)\cdot\omega\geq0\}\]
of $\partial\Omega$. We associate to $\partial\Omega_{\pm,\omega}$ the part of the lateral boundary $\Sigma$ given by $\Sigma_{\pm,\omega}=(0,T)\times\partial\Omega_{\pm,\omega}$. From now on we fix $\omega_0\in \mathbb S^{n-1}$ and we consider $F=(0,T)\times F'$ (resp $G=(0,T)\times G'$) with $F'$ (resp $G'$) a closed  neighborhood of $\partial\Omega_{+,\omega_0}$ (resp $\partial\Omega_{-,\omega_0}$) in $\partial\Omega$.

From now on we denote by $\Box$  the differential operator $\partial_t^2-\Delta_x$. According to  \cite[Proposition 4]{Ki2}, we can extend the trace maps
\[\tau_{0,1}v=v_{|\Sigma},\ \tau_{0,2}v=v_{|t=0},\ \tau_{0,3}v=\partial_tv_{|t=0},\quad v\in\mathcal C^\infty(\overline{Q})\]
on $H_{\Box}(Q)=\{u\in L^2(Q):\ \Box u\in L^2(Q)\}$. Then we define 
\[\mathcal H_{F}(\partial Q):=\{(\tau_{0,1}u,\tau_{0,3}u):\ u\in H_{\Box}(Q),\ \tau_{0,2}u=0,\ \textrm{supp}\tau_{0,1}u\subset F\}.\]
We refer to \cite[Section 2]{Ki2} (see also Section 2) for more details about these spaces and the definition of $\norm{.}_{\mathcal H_F(\partial Q)}$.
In view \cite[Section 2]{Ki2}, we can  associate to \eqref{eq1} with $v_0=0$ the boundary operator
\bel{a1}
B_q:  \mathcal H_{F}(\partial Q) \ni (g,v_1)  \mapsto (\partial_\nu u_{|G},u_{|t=T})
\ee
where $u$ solves \eqref{eq1} with $v_0=0$. We refer to \cite[Proposition 2]{Ki2} (see also Section 2) for a more rigorous definition of this operator.
In Section 2, we prove that for every $q_1,\ q_2 \in L^{\infty}(Q)$  the operator 
\[B_{q_1}-B_{q_2}:\mathcal H_{F}(\partial Q)\rightarrow L^2(G)\times H^1(\Omega)\]
is bounded. Then our main result can be  stated as follows.

\begin{Thm}\label{thm1} 
Let $p>n+1$ and $q_1,\ q_2 \in W^{1,p}(Q)$ . Assume that the conditions 
\begin{equation}\label{thmma}q_1(t,x)=q_2(t,x),\quad (t,x)\in\Sigma\end{equation}
\[\norm{q_1}_{W^{1,p}(Q)}+\norm{q_2}_{W^{1,p}(Q)}\leq M\]
are fulfilled.
Then,  we have
\begin{equation}\label{thm1a} \norm{q_1-q_2}_{L^2(Q)}\leq h\left(\norm{B_{{q}_1}-B_{{q}_2}}\right)\end{equation}
with 
\[h(\gamma)=\left\{\begin{array}{l}C\frac{\gamma}{\gamma^*},\quad \gamma\geq\gamma^*,\\
C\ln(\abs{\ln \gamma})^{-\frac{1}{2}},\quad 0<\gamma<\gamma^*,\\ 0,\quad \gamma=0.\end{array}\right.\]
Here $\norm{B_{{q}_1}-B_{{q}_2}}$ stands for the norm of $B_{{q}_1}-B_{{q}_2}$ as an element of $\mathcal B(\mathcal H_{F}(\partial Q);L^2(G)\times H^1(\Omega))$. Moreover, $C$ is a positive constant depending on $n$, $p$,  $M$, $\Omega$, $T$, $F'$, $G'$ and $\gamma_*=e^{-e^{A{R}_2}}$, with $A$ and $R_2$ two constants introduced in Section 5 which depend  on  $n$, $p$, $M$, $T$, $\Omega$, $F'$, $G'$.\end{Thm}

Let us observe that  this stability estimate is the first that is stated with the data considered in \cite{Ki2}, where uniqueness is stated with conditions that seems to be one of the weakest so far. Moreover, it appears that with the paper of \cite{A}, this paper is the first where stability is stated for global determination of general time dependent potential appearing in wave equation from boundary measurements.

The main tools in our analysis are suitable  geometric optics (GO in short) solutions,  Carleman estimates and results of stability in  analytic continuation. More precisely, following the approach of \cite{Ki2} combined with arguments used by \cite{CKS1,BJY,HW} (see also \cite{BU,KSU,NS} for the original aproach in the case of elliptic equations), we consider suitable geometric optics  solutions for our problem associated to Carleman estimate with linear weight. In contrast to \cite{Ki2}, we recover the time dependent potential not from its Fourier transform but from its  light-ray transform (see the proof of Theorem \ref{thm1}). This approach make it possible to derive stability even in the case $n=2$. Note also that in contrast to \cite{Ki2}, for the stability issue it is necessary to consider GO lying in $H^2(Q)$ (and not only in $H^1(Q)$).

\subsection{Outline}

This paper is organized as follows. In Section 2 we treat the direct problem. We recall  some properties of solutions of \eqref{eq1} and we give a result of smoothing for the difference of boundary operators $B_{q_1}-B_{q_2}$  associated to  this problem. In Section 3, using some results of \cite{Ch}, \cite{Ho1} and \cite{Ho2}, we build  GO   solutions, similar to \cite{Ki} and \cite{BJY}, associated to \eqref{eq1} and lying in $H^2(Q)$. In Section 4,  we recall some results of \cite{Ki} about  Carleman estimates for the wave equation with linear weight and GO solutions vanishing on parts of the boundary. Then combining these tools with the GO solutions of Section 3 we prove Theorem \ref{thm1}.
\vspace{5mm}
\ \\
\textbf{Acknowledgements}. The author would like to thank Mourad Bellassoued and Eric Soccorsi   for their   remarks and  suggestions.

\section{Functional spaces}
In this section following \cite{Ki2} we recall some properties of the IBVP \eqref{eq1}. According to  \cite[Proposition 1]{Ki2}, for any $(g,v_1)\in \mathcal H_{F}(\partial Q)$  the IBVP \eqref{eq1} with $q=v_0=0$ admits a unique solution $\mathcal P_0(g,v_1)$ and we can define $\norm{\cdot}_{\mathcal H_F(\partial Q)}$ by 
\[\norm{(g,v_1)}_{\mathcal H_F(\partial Q)}=\norm{\mathcal P_0(g,v_1)}_{L^2(Q)}.\]
 Applying \cite[Proposition 4]{Ki2}, we can extend the map 
\[\tau_{1,1}v=\partial_\nu v_{|\Sigma},\quad \tau_{1,2}v= v_{|t=T},\quad v\in\mathcal C^\infty(\overline{Q})\]
on $H_\Box (Q)$. Then, in light of \cite[Proposition 2]{Ki2}, we can define the boundary operator $$B_q: \mathcal H_F(\partial Q)\ni(g,v_1)\mapsto(\tau_{1,1}u_{|G},\tau_{1,2}u)$$ 
with $u\in L^2(Q)$ the unique weak solution of the IBVP \eqref{eq1} with $v_0=0$. Moreover, in view of \cite[Proposition 2]{Ki2}, $B_q$ is bounded from $\mathcal H_F(\partial Q)$ to $ H^{-3}(0,T; H^{-\frac{1}{2}}(G'))\times H^{-2}(\Omega)$.

Now consider the operator $B_{{q}_1}-B_{{q}_2}$ for $q_1,\ q_2 \in L^{\infty}(Q)$. We have the following smoothing  result.
\begin{prop}\label{p7} Let $q_1,\ q_2 \in L^{\infty}(Q)$. Then the operator $B_{{q}_1}-B_{{q}_2}$ is a bounded operator from $\mathcal H(\partial Q)$ to $L^2(G)\times H^1(\Omega)$.

\end{prop}
 \begin{proof} Let $u_1,u_2$ be respectively the unique solution of  the IBVP \eqref{eq1} for $q=q_1$ and $q=q_2$ and $v_0=0$. Then $u=u_1-u_2$ solves
\[\left\{ \begin{array}{rcll} \partial_t^2u-\Delta u+q_1u& = & (q_2-q_1)u_2, & (t,x) \in Q ,\\ 

u_{\vert t=0}=\partial_t u_{\vert t=0}&=&0,\\
u_{\vert\Sigma}& = & 0.&\end{array}\right.
\]Since $(q_2-q_1)u_2\in L^2(Q)$, in view of Theorem A.2 in \cite{BCY}  (see also Theorem 2.1 in \cite{LLT} for $q=0$), $u\in \mathcal C^1([0,T];L^2(\Omega))\cap \mathcal C([0,T];H^1_0(\Omega))$ with $\partial_\nu u\in L^2(\Sigma)$. Moreover, we have the following energy estimate

\[\norm{u}_{\mathcal C^1([0,T];L^2(\Omega))}+\norm{u}_{\mathcal C([0,T];H^1_0(\Omega))} +\norm{\partial_\nu u}_{L^2(\Sigma)}\leq C\norm{q_1-q_2}_{L^\infty(Q)}\norm{u_2}_{L^2(Q)}.\]
It follows  $\tau_{1,1}u_{|G}\in L^2(G)$, $\tau_{1,2}u\in H^1(\Omega)$ with
\[\norm{ \tau_{1,1}u}_{ L^2(G)}+\norm{\tau_{1,2}u}_{H^1(\Omega)}\leq C\norm{(g,v_1)}_{\mathcal H_F(\partial Q)},\]
where $C$ depends  on  $\Omega$, $T$ and $M\geq \norm{q_1}_{L^\infty(Q)}+\norm{q_2}_{L^\infty(Q)}$.
Finally, we complete the proof by recalling that
\[(\tau_{1,1}u_{|G},\tau_{1,2}u) =(\tau_{1,1}{u_1}_{|G},\tau_{1,2}u_1)-(\tau_{1,1}{u_2}_{|G},\tau_{1,2}u_2)=(B_{{q}_1}-B_{{q}_2})(g,v_1).\]\end{proof}

\section{Smooth geometric optics solutions without boundary conditions}

The goal of this section is to build GO solutions $u\in H^2(Q)$ associated to the equation 
\begin{equation}\label{eqGO1}\partial_t^2u-\Delta u+q(t,x)u=0\quad \textrm{on } Q.\end{equation}
More precisely, for $\lambda>0$, $\omega\in\mathbb S^{n-1}=\{y\in\R^n:\ |y|=1\}$, $\phi\in\mathcal C^\infty(\R^n)$ we consider solutions of this equation of the form
\begin{equation}\label{CGO1} u=e^{-\lambda(t+x\cdot\omega)}(\chi(t,x)+w(t,x))\end{equation}
with $u\in H^2(Q)$ and $\chi(t,x)=\phi(x+t\omega)$.
Here $w$ is the remainder term  in the asymptotic expansion of $u$ with respect to $\lambda$ and we have
\[\norm{w}_{H^1(Q)}\leq \frac{C}{\lambda}\]
with $C>0$ independent of $\lambda$.  In order to build such GO solutions, we first introduce some well known results of H\"ormander about solutions of  PDEs with constant coefficients  of the form $P(D)u=f$ on $Q$ with $P$ a polynomial of $n+1$ variables with complex  valued coefficients and $D=-i(\partial_{t},\partial_{x})$. 
\subsection{Solutions of PDEs with constant coefficients }
We start this subsection by recalling some properties of  solutions of PDEs of the form $P(D)u=f$ with constant coefficients.
For $P$  a polynomial of $n+1$ variable, let $\tilde{P}$ be defined by
\[\tilde{P}(\mu,\xi)=\left(\sum_{k\in\mathbb N}\sum_{\alpha\in\mathbb N^{n}}\abs{ \partial^k_\mu\partial^\alpha_{\xi}P(\mu,\xi)}^2\right)^{\frac{1}{2}},\quad \mu\in\R,\ \xi\in\R^n.\]

\begin{Thm}\label{ttt1}\emph{(Theorem 7.3.10, \cite{Ho1})}For every  $P\neq0$ polynomial of $n+1$ variables  one can find a distribution of finite order $E_P\in D'(\R^{1+n})$ such that $P(D)E_P=\delta$.\end{Thm}
Such distributions $E_P$ are called fundamental solutions of $P$. Note that
\[E_P*(P(D)u)=u,\quad u\in \mathcal E'(\R^{1+n}),\]
\[P(D)(E_P*f)=f,\quad f\in \mathcal E'(\R^{1+n}),\]
 where $\mathcal E'(\R^{1+n})$ is the set of distributions with compact support. Thus, for all $f\in \mathcal E'(\R^{1+n})$,  $u=E_P*f$ is a solution of $P(D)u=f$ . Let us give some information about the regularity of such a solution. For this purpose we need the following definitions introduced in \cite{Ho2}.
 \begin{defi} A positive function $\kappa$ defined in $\R^{1+n}$ will be called a temperate weight function if there exist  positive constants $C$ and $N$  such that
 \[\kappa(\zeta+\eta)\leq C(1+\abs{\zeta})^{N}\kappa(\eta),\quad \zeta,\eta\in\R^{1+n}.\]
 The set of all such functions $\kappa$ will be denoted by $\mathcal K$.\end{defi}
 Notice that, for all polynomial of $n+1$ variables $P$, $\tilde{P}\in \mathcal K$.
 \begin{defi} If $\kappa\in \mathcal K$ and $1\leq p\leq \infty$, we denote by $B_{p,\kappa}$ the set of all temperate distribution $u\in\mathcal S'(\R^{1+n})$ such that $\hat{u}$ is a function and
 \[\norm{u}_{p,\kappa}=\left(\frac{1}{(2\pi)^{1+n}}\int_{\R^{n}}\int_\R\abs{\kappa(\mu,\xi)\hat{u}(\mu,\xi)}^pd\mu d\xi\right)^{\frac{1}{p}}<\infty.\]
When $p=\infty$ we shall interpret $\norm{u}_{p,\kappa}$ as ess. sup$\abs{\kappa(\mu,\xi)\hat{u}(\mu,\xi)}$.
 We denote by $B_{p,\kappa}^{loc}$ the set of $u\in\mathcal S'(\R^{1+n})$ such that for all $\chi\in\mathcal C_0^{\infty}(\R^{1+n})$ we have $\chi u\in B_{p,\kappa}$.\end{defi}
 \begin{rem} \label{r1}  Let 
  \[\kappa_1(\mu,\xi)=(1+\abs{(\mu,\xi)}^2)^{\frac{1}{2}},\quad \tau\in\R,\ \eta\in\R^n.\]
Then, in view of Example 10.1.2 of \cite{Ho2}, one can easily show that $\kappa_1\in \mathcal K$  and $B_{2,\kappa_1}=H^{1}(\R^{1+n})$. \end{rem}
  \begin{rem} \label{r2}In view of Theorem 10.1.12 of \cite{Ho2}, for $\kappa'_1,\kappa'_2\in\mathcal K$, $\kappa=\kappa'_1\cdot\kappa'_2$, $u_1\in B_{p,\kappa'_1}\cap \mathcal E'(\R^{1+n})$ and $u_2\in B_{\infty,\kappa'_2}$, we have $u_1*u_2\in B_{p,\kappa}$ and
   \begin{equation}\label{f2}\norm{u_1*u_2}_{B_{p,\kappa}}\leq \norm{u_1}_{B_{p,\kappa'_1}}\norm{u_2}_{B_{\infty,\kappa'_2}}.\end{equation}\end{rem}
    
\begin{Thm}\emph{(Theorem 10.21, \cite{Ho2})} Every polynomial of $n+1$ variables $P\neq0$, has a fundamental solution $E_P\in B_{\infty,\tilde{P}}^{loc}$ such that $\frac{E_P}{\cosh(\abs{(t,x)})}\in B_{\infty,\tilde{P}}$ and 
   \begin{equation}\label{f3}\norm{\frac{E_P}{\cosh(\abs{(t,x)})}}_{B_{\infty,\tilde{P}}}\leq C\end{equation}
   with $C>0$ a constant depending only on the degree of $P$.
\end{Thm}
 Such a fundamental solution will be denoted by regular fundamental solution. Let us remark that in our construction of  geometric optics solutions we need to consider an operator $E$ such that $P(D)E=\textrm{Id}$ for some polynomial of $n+1$ variables $P\neq0$. Using the properties of regular fundamental solutions, H\"ormander proved in Theorem 10.3.7 of \cite{Ho2} that such a operator exists and it is a bounded operator of $L^2(X)$ for $X$ a bounded open set of $\R^{1+n}$. In contrast to elliptic equations and parabolic equations (see  Subsection 2.1 and 3.6 of \cite{Ch}), for hyperbolic equations  we can not build GO lying in $H^2(Q)$ by applying the result of  \cite{Ho2}. What we can actually build from this result is GO lying in $H^1(Q)$ (e.g. \cite[Proposition 3]{Ki}) . Therefore, we need to consider the following.

 \begin{prop}\label{tt5} Let  $P\neq0$ be a  polynomial of $n+1$ variables. Then there exists  an operator 
  \[E:\ H^{1}(Q)\to H^{1}(Q)\]
   such that:
 \begin{enumerate}
 \item $P(D)Ef=f,\quad f\in H^{1}(Q)$,
 \item for all  polynomial of $n+1$ variables $S$ such that $\frac{\tilde{S}}{\tilde{P}}$ is bounded, we have $S(D)E\in B(H^{1}(Q))$,  and
  \begin{equation}\label{pp5a}\norm{S(D)E}_{B(H^{1}(Q))}\leq C\sup_{(\mu,\xi)\in\R\times\R^{n}}\frac{\abs{S(\mu,\xi)}}{\tilde{P}(\mu,\xi)},\quad k=0,1,\end{equation}
  where $C>0$ depends only on the degree of $P$, $\Omega$ and $T$.
  \end{enumerate}
\end{prop}
  
 \begin{proof} Let $f\in H^1(Q)$.  In view of   Theorem 2.2 and 8.1 in Chapter 1 of \cite{LM1}, there exists an extension operator $\mathcal E\in  B\left(H^1(Q), H^{1}(\R^{1+n})\right)$ such that $\mathcal Ef_{\vert Q}=f$. Here we consider the extension operator $\mathcal E$ introduced by \cite{LM1}. Set $\chi\in\mathcal C^\infty_0(\R^{1+n})$ and $R>0$ such that $\chi=1$ on a neighborhood of $\overline{Q}$  and supp$\chi\subset B_R$ with $B_R$ the ball of radius $R$ and of center $0$ of $\R^{1+n}$.  Let $E_P$ be a regular fundamental solution of $P$. Now consider the operator
 \[E:f\longmapsto \left(E_P*(\chi \mathcal Ef)\right)_{\vert Q}.\]
 Clearly we have 
 \[P(D)E_P*(\chi \mathcal Ef)=\chi \mathcal Ef\]
 and it follows that
 \[P(D)Ef=\left(\chi \mathcal Ef\right)_{\vert Q}=f\]
 which proves (1). Now let us show (2). For this purpose, let $\psi\in\mathcal C^\infty_0(\R^{1+n})$ be such that $\psi=1$ on the closure of $B_R-B_R=\{x-y:\ x,y\in B_R\}$ and notice that
 \begin{equation}\label{pp5b}\left(E_P*(\chi \mathcal Ef)\right)_{\vert Q}=\left((\psi E_P)*(\chi \mathcal Ef)\right)_{\vert Q}.\end{equation}
Note that 
\[\abs{\mathcal F(S(D)\psi E_P)(\mu,\xi)}\leq \frac{\abs{S(\mu,\xi)}}{\tilde{P}(\mu,\xi)} \tilde{P}(\mu,\xi)\abs{\mathcal F\left(\psi \cosh(\abs{(t,x)}) \frac{E_P}{\cosh(\abs{(t,x)})}\right)(\mu,\xi)},\ \mu\in\R,\ \xi\in\R^n.\]
Then, since $\psi \cosh(\abs{(t,x)})\in\mathcal C^\infty_0(\R^{1+n})$, from Lemma 2.1 of \cite{Ch} we deduce that \[\psi \cosh(\abs{(t,x)}) \frac{E_P}{\cosh(\abs{(t,x)})}\in B_{\infty,\tilde{P}}\] and
\[\norm{\psi \cosh(\abs{(t,x)}) \frac{E_P}{\cosh(\abs{(t,x)})}}_{B_{\infty,\tilde{P}}}\leq C_1\norm{\frac{E_P}{\cosh(\abs{(t,x)})}}_{B_{\infty,\tilde{P}}}\leq C'\]
with $C'>0$ a constant depending only on the degree of $P$ and $\chi$. It follows that $S(D)\psi E_P\in B_{\infty,1}$ and
\[\norm{S(D)\psi E_P}_{B_{\infty,1}}\leq  C'\sup_{(\mu,\xi)\in\R\times\R^{n}}\frac{\abs{S(\mu,\xi)}}{\tilde{P}(\mu,\xi)}.\]
In view of Remark \ref{r2}, since $\chi \mathcal Ef\in H^1(\R^{1+n})=B_{2,\kappa_1}$ with $\kappa_1$ introduced in Remark \ref{r1}, we have  $S(D)(\psi E_P)*(\chi \mathcal Ef)=(S(D)\psi E_P)*(\chi \mathcal Ef)\in B_{2,\kappa_1}$ and
\[\begin{aligned}\norm{S(D)(\psi E_P)*(\chi \mathcal Ef)}_{H^1(\R^{1+n})}&=\norm{S(D)(\psi E_P)*(\chi \mathcal Ef)}_{B_{2,\kappa_1}}\\
\ &\leq \norm{S(D)\psi E_P}_{B_{\infty,1}}\norm{\chi \mathcal Ef}_{H^1(\R^{1+n}))}\\
\ &\leq C\sup_{(\mu,\xi)\in\R\times\R^{n}}\frac{\abs{S(\mu,\xi)}}{\tilde{P}(\mu,\xi)}\norm{f}_{H^1(Q)}\end{aligned}\]
with $C>0$ a constant depending only on the degree of $P$, $\chi$, $\Omega$ and $T$. Thus, in view of \eqref{pp5b}, we have $S(D)Ef\in H^1(Q)$ and 
\[\norm{S(D)Ef}_{H^1(Q)}\leq \norm{(\psi E_P)*(\chi \mathcal Ef)}_{H^{1}(Q)}\leq C\sup_{(\mu,\xi)\in\R\times\R^{n}}\frac{\abs{S(\mu,\xi)}}{\tilde{P}(\mu,\xi)}\norm{f}_{H^1(Q)}.\]

\end{proof}
 
 Armed with this result, we are now in position to build GO of the form \eqref{CGO1} lying in $H^2(Q)$.
  
\subsection{Construction of geometric optics solutions}
The goal of this subsection is to apply the results of the previous subsection in order to build geometric optics of the form \eqref{CGO1}.
For this purpose, for all $s\in\R$ and all $\omega\in\mathbb S^{n-1}$, we consider the operators $P_{s,\omega}$ defined by $P_{s,\omega}=e^{-s(t+x\cdot\omega)}\Box e^{s(t+x\cdot\omega)}$. One can check that
\[P_{s,\omega}=p_{s,\omega}(D_t,D_x)=\Box +2s(\partial_t-\omega\cdot \nabla_x)\]
with $D_t=-i\partial_t$, $D_x=-i\nabla_x$ and $p_{s,\omega}(\mu,\xi)=-\mu^2+\abs{\xi}^2+i\lambda (\mu-\omega\cdot\xi)$, $\mu\in\R$, $\xi\in\R^n$.
 Applying Theorem \ref{tt5} to $P_{-\lambda,\omega}$ we obtain the following intermediate result.

\begin{lem}\label{ppl1} For every $\lambda>1$ and $\omega\in \mathbb S^{n-1}$ there exists a bounded operator $E_{\lambda,\omega}:\ H^1(Q)\to H^1(Q)$ such that:
\begin{equation}\label{ppl1a}P_{-\lambda,\omega} E_{\lambda,\omega}f=f,\quad f\in H^1(Q),\end{equation}
\begin{equation}\label{ppl1b} \norm{E_{\lambda,\omega}}_{\mathcal B(L^2(Q))}\leq C\lambda^{-1},\quad f\in L^2(Q),\end{equation}
\begin{equation}\label{ppl1c} E_{\lambda,\omega}\in \mathcal B(H^1(Q);H^2(Q))\quad \textrm{and}\quad \norm{E_{\lambda,\omega}}_{\mathcal B(H^1(Q);H^2(Q))}\leq C\end{equation}
with $C>$ depending only on $T$ and $\Omega$.\end{lem}
\begin{proof} In light of Theorem \ref{tt5}, there exists a bounded operator $E_{\lambda,\omega}:\ H^1(Q)\to H^1(Q)$, defined from a fundamental solution associated to $P_{-\lambda,\omega}$,  such that \eqref{ppl1a} is fulfilled. In addition, for all differential operator  $Q(D_t,D_x)$ such that ${Q(\mu,\xi)\over \tilde {p}_{-\lambda,\omega}(\mu,\xi)}$ is bounded, we have $Q(D_t,D_x)E_{\lambda,\omega}\in\mathcal B(H^1(Q))$ and
\begin{equation}\label{ppl1d}\norm{Q(D_t,D_x)E_{\lambda,\omega}}_{\mathcal B(H^1(Q))}\leq C\sup_{(\mu,\xi)\in\R^{1+n}}{|Q(\mu,\xi)|\over \tilde {p}_{-\lambda,\omega}(\mu,\xi)},\end{equation}
where $C>0$ depends only on $\Omega$, $T$ and $\tilde {p}_{-\lambda,\omega}$ is given by
\[\tilde{p}_{-\lambda,\omega}(\mu,\xi)=\left(\sum_{k\in\mathbb N}\sum_{\alpha\in\mathbb N^n}|\partial^k_\mu\partial^\alpha_\xi p_{-\lambda,\omega}(\mu,\xi)|^2\right)^{{1\over2}}.\]
Note that $\tilde{p}_{-\lambda,\omega}(\mu,\xi)\geq \abs{\im \partial_\mu p_{-\lambda,\omega}(\mu,\xi)}=2\lambda$. Therefore, \eqref{ppl1d} implies
\[\norm{E_{\lambda,\omega}}_{\mathcal B(H^1(Q))}\leq C\sup_{(\mu,\xi)\in\R^{1+n}}{1\over \tilde {p}_{-\lambda,\omega}(\mu,\xi)}\leq C\lambda^{-1}\]
and \eqref{ppl1b} is fulfilled. In a same way, we have $\tilde{p}_{-\lambda,\omega}(\mu,\xi)\geq \abs{\re \partial_\mu p_{-\lambda,\omega}(\mu,\xi)}=2|\mu|$ and $\tilde{p}_{-\lambda,\omega}(\mu,\xi)\geq \abs{\re \partial_{\xi_i} p_{-\lambda,\omega}(\mu,\xi)}=2|\xi_i|$, $i=1,\ldots,n$ and $\xi=(\xi_1,\ldots,\xi_n)$. Therefore, in view of condition (2) of Theorem \ref{tt5}, for all $h\in H^1(Q)$, we have $\partial_tE_{\lambda,\omega}h$, $\partial_{x_1}E_{\lambda,\omega}h,\ldots,\partial_{x_n}E_{\lambda,\omega}h\in H^1(Q)$ with
\[\norm{\partial_tE_{\lambda,\omega}h}_{H^1(Q)}+\sum_{j=1}^n\norm{\partial_{x_1}E_{\lambda,\omega}h}_{H^1(Q)}\leq C\left(\sup_{(\mu,\xi)\in\R^{1+n}}{|\mu|+|\xi_1|+\ldots+|\xi_n|\over \tilde {p}_{-\lambda,\omega}(\mu,\xi)}\right)\norm{h}_{H^1(Q)}\leq C(n+1)\norm{h}_{H^1(Q)}\]
Thus, we get $E_{\lambda,\omega}\in \mathcal B(H^1(Q);H^2(Q))$ with
\[\norm{E_{\lambda,\omega}}_{\mathcal B(H^1(Q);H^2(Q))}\leq C\sup_{(\mu,\xi)\in\R^{1+n}}{|\mu|+|\xi_1|+\ldots+|\xi_n|\over \tilde {p}_{-\lambda,\omega}(\mu,\xi)}\leq C(n+1)\]
and \eqref{ppl1c} is proved.\end{proof}

Armed with this result, we are now in position to build  geometric optics solutions of the form \eqref{CGO1} lying in $H^2(Q)$.
  \begin{Thm}\label{p2} Let $q\in W^{1,p}(Q)$, with $p>n+1$, be such that $\norm{q}_{W^{1,p}(Q)}\leq M$,  $\omega\in\mathbb S^{n-1}$, $\lambda>1$.  Then, there exists $\lambda_0>1$ such that for $\lambda\geq \lambda_0\norm{\chi}_{H^3(Q)}$ the equation \eqref{eqGO1} admits a solution $u\in H^2(Q)$ of the form \eqref{CGO1} with
\begin{equation}\label{p2a}\norm{w}_{H^k(Q)}\leq C\lambda^{k-2}\norm{\chi}_{H^3(Q)},\quad k=1,2,\end{equation}
where $C$ and $\lambda_0$ depend on $\Omega$, $T$, $M$, $n$, $p$.
\end{Thm}
\begin{proof} We start by recalling that
\[\Box e^{-\lambda(t+x\cdot\omega)}\chi(t,x)=e^{-\lambda(t+x\cdot\omega)}\Box \chi(t,x),\quad (t,x)\in Q.\]
Thus, $w$ should be a solution of
\begin{equation}\label{p2b}\partial_t^2w-\Delta w-2\lambda(\partial_t-\omega\cdot\nabla_{x})w=-\left((\Box+q) \chi(t,x)+qw\right).\end{equation}
Note that since $q \in W^{1,r}(Q)$ with $r>n+1$, using the Sobolev embedding theorem (e.g. \cite[Theorem 1.4.4.1]{Gr}) and H\"older inequality, one can check for all $w\in H^1(Q)$, $qw\in H^1(Q)$ with 
\begin{equation}\label{p2b'}\norm{qw}_{H^1(Q)}\leq CM\norm{w}_{H^1(Q)}\end{equation}
with $C$ depending only on $T$ and $\Omega$, $n$, $p$. Therefore, according to Lemma \ref{ppl1}, we can define $w$ as a solution of the equation $$w=-E_{\lambda,\omega}\left((\Box+q) \chi(t,x)+qw\right),\quad w\in H^1(Q)$$ with $E_{\lambda,\omega}\in\mathcal B(H^1(Q))$ given by Lemma \ref{ppl1}.
For this purpose, we will use a standard fixed point argument associated to the map
\[\begin{array}{rccl} \mathcal G: & H^1(Q) & \to & H^1(Q), \\
 \ \\ & F & \mapsto &-E_{\lambda,\omega}\left[(\Box+q) \chi(t,x)+qF\right]. \end{array}\]
Indeed, in view of \eqref{ppl1b}, fixing $M_1>0$, there  exists $\lambda_0>1$ such that for $\lambda\geq \lambda_0\norm{\chi}_{H^3(Q)}$ the map $\mathcal G$ admits a unique fixed point $w$ in $\{u\in H^1(Q): \norm{u}_{H^1(Q)}\leq M_1\}$. In addition, condition \eqref{ppl1b}-\eqref{ppl1c} imply that $w\in H^2(Q)$ fulfills \eqref{p2a}. This completes the proof.
\end{proof}

\section{Stability estimate}
This section is devoted to the proof of Theorem \ref{thm1}. We start by collecting some tools constructed in \cite{Ki2} that will play an important role in the proof of Theorem \ref{thm1}.

\subsection{Carleman estimate and geometric optics solutions vanishing on parts of the boundary}
The goal of this section is to recall some useful tools for the proof of Theorem \ref{thm1}. We first consider the following Carleman estimates.

\begin{Thm}\label{c1}\emph{(Theorem 2, \cite{Ki2})}  Let $q\in L^\infty(Q)$ and  $u\in\mathcal C^2(\overline{Q})$.  If $u$ satisfies the condition 
 \begin{equation}\label{ttc1}u_{\vert \Sigma}=0,\quad u_{\vert t=0}=\partial_tu_{\vert t=0}=0\end{equation}
then there exists $\lambda_1>1$ depending only on  $\Omega$, $T$ and $M\geq \norm{q}_{L^\infty(Q)}$ such that the estimate
\begin{equation}\label{c1a}\begin{array}{l}\lambda \int_\Omega e^{-2\lambda( T+\omega\cdot x)}\abs{\partial_tu_{\vert t=T}}^2dx+\lambda\int_{\Sigma_{+,\omega}}e^{-2\lambda(t+\omega\cdot x)}\abs{\partial_\nu u}^2\abs{\omega\cdot\nu(x) } d\sigma(x)dt+\lambda^2\int_Qe^{-2\lambda(t+\omega\cdot x)}\abs{u}^2dxdt\\
\leq C\left(\int_Qe^{-2\lambda(t+\omega\cdot x)}\abs{(\partial_t^2-\Delta+q)u}^2dxdt+ \lambda^3\int_\Omega e^{-2\lambda(T+\omega\cdot x)}\abs{u_{\vert t=T}}^2dx+\lambda\int_\Omega e^{-2\lambda(T+\omega\cdot x)}\abs{\nabla_xu_{\vert t=T}}^2dx\right)\\
\ \ \ +C\lambda\int_{\Sigma_{-,\omega}}e^{-2\lambda(t+\omega\cdot x)}\abs{\partial_\nu u}^2\abs{\omega\cdot\nu(x) }d\sigma(x)dt\end{array}\end{equation}
holds true for $\lambda\geq \lambda_1$  with $C$ and $\lambda_1$ depending only on  $\Omega$, $T$ and $M\geq \norm{q}_{L^\infty(Q)}$.
\end{Thm}

We precise  that  this Carleman estimate has been proved in \cite{Ki2} following some arguments of \cite{BJY}. 

 From now on, for all $y\in\mathbb S^{n-1}$ and all  $r>0$, we set
\[\partial\Omega_{+,r,y}=\{x\in\partial\Omega:\ \nu(x)\cdot y>r\},\quad\partial\Omega_{-,r,y}=\{x\in\partial\Omega:\ \nu(x)\cdot y\leq r\}\]
and $\Sigma_{\pm,r,y}=(0,T)\times \partial\Omega_{\pm,r,y}$. Here and in the remaining of this text we always assume, without mentioning it, that $y$ and $r$ are chosen in such way that $\partial\Omega_{\pm,r,\pm y}$ contain  a non-empty relatively open subset of $\partial\Omega$
. Without lost of generality we can assume that there exists $0<\epsilon<1$ such that for all $\omega\in\{y\in\mathbb S^{n-1}:|y-\omega_0|\leq\epsilon\} $ we have $\partial\Omega_{-,\epsilon,-\omega}\subset F'$. We consider $u\in H_{\Box}(Q)$ satisfying
\begin{equation}
\label{(5.1)}
\left\{
\begin{array}{l}
(\partial_t^2-\Delta  +q(t,x))u=0\ \ \textrm{in }  Q,
\\
u_{\vert t=0}=0,
\\
u=0,\quad \ \textrm{on } \Sigma_{+,\epsilon/2,-\omega},
\end{array}\right.\end{equation}
 of the form
\bel{CGO1a}
u(t,x)=e^{\lambda (t+\omega\cdot x)}\left( \chi(t,x)+z(t,x) \right),\quad (t,x)\in Q,
\ee
where $\omega\in\{y\in\mathbb S^{n-1}:|y-\omega_0|\leq\epsilon\} $, $\chi(t,x)=\phi(x+t\omega)$, $z \in e^{-\lambda (t+\omega\cdot x)}H_{\Box}(Q)$ fulfills: $z(0,x)=-\chi(0,x)$ , $x\in\Omega$, $z=-\chi(t,x)$ on  $\Sigma_{+,\epsilon/2,-\omega}$ and

\bel{CGO1b}
\| z \|_{L^2(Q)}\leq C\lambda^{-\frac{1}{2}}\norm{\chi}_{H^2(Q)}
\ee
with $C$ depending on $F'$, $\Omega$, $T$, $p$, $n$ and $M$. Since $\Sigma\setminus F\subset \Sigma_{+,\epsilon,-\omega}$  and since $\Sigma_{+,\epsilon/2,-\omega}$ is a neighborhood of $\Sigma_{+,\epsilon,-\omega}$ in $\Sigma$,  it is clear that condition \eqref{(5.1)} implies $(\tau_{0,1}u,\tau_{0,3}u)\in\mathcal H_F(\partial Q)$  (recall that for $v\in\mathcal C^\infty(\overline{Q})$, $\tau_{0,1}v=v_{|\Sigma}$, $\tau_{0,3}v=\partial_tv_{|t=0}$).  Repeating some arguments of \cite[Theorem 3]{Ki2}, we prove the following.

\begin{Thm}\label{tt1}  Let $q\in L^\infty(Q)$. For all $\lambda\geq \lambda_1$, with $\lambda_1$ the constant of Theorem \ref{c1}. Then, there exists a solution $u\in H_{\Box}(Q)$ of \eqref{(5.1)} of the form \eqref{CGO1a} with $z$ satisfying \eqref{CGO1b}. \end{Thm}
\begin{proof} Note first that $z$ must satisfy
\begin{equation}
\label{w1}
\left\{
\begin{array}{l}z\in L^2(Q) \\
(\partial_t^2-\Delta+q) (e^{\lambda(t+\omega\cdot x)}z)=-e^{\lambda(t+\omega\cdot x)}(\Box +q)\chi(t,x)\ \ \textrm{in }Q
\\
z(0,x)=-\chi(0,x), \quad   x\in\Omega,
\\
z=-\chi(t,x)\quad \textrm{on } \Sigma_{+,\epsilon/2,-\omega}.
\end{array}\right.\end{equation}
Let $\psi\in\mathcal C^\infty_0(\R^n)$ be such that   supp$\psi\cap\partial\Omega\subset \{x\in\partial\Omega:\ \omega\cdot\nu(x)<-\epsilon/3\}$ and $\psi=1$ on $\{x\in\partial\Omega:\ \omega\cdot\nu(x)<-\epsilon/2\}=\partial\Omega_{+,\epsilon/2,-\omega}$. Choose $v_-(t,x)=-e^{\lambda(t+\omega\cdot x)}\psi(x)\chi(t,x)$,
$v(t,x)=-e^{\lambda(t+\omega\cdot x)}(\Box +q)\chi(t,x)$ and $v_0(x)=-e^{\lambda\omega\cdot x}\chi(0,x)$. Then, in view of \cite[Lemma 3]{Ki2}, there exists $w\in H_\Box(Q)$ such that
\[
\left\{
\begin{array}{ll}
(\partial_t^2-\Delta+q) w=v(t,x)=-e^{\lambda(t+\omega\cdot x)}(\Box +q)\chi(t,x)&  \mbox{in}\; Q,
\\
w(0,x)=v_0(x)=-e^{\lambda\omega\cdot x}\chi(0,x), &  x\in\Omega,
\\
w(t,x)=v_-(t,x)=-e^{\lambda(t+\omega\cdot x)}\psi(x)\chi(t,x),& (t,x)\in\Sigma_{-,\omega}.
\end{array}
\right.\]
Then, for $z=e^{-\lambda(t+\omega\cdot x)} w$ condition \eqref{w1} will be fulfilled. Moreover,  in light of \cite[Lemma 3]{Ki2}, we have  \eqref{CGO1b}.\end{proof}

Armed with these results we are now in position to complete the proof of Theorem \ref{thm1}.

\subsection{Proof of Theorem \ref{thm1}}
In this subsection we complete  the proof of Theorem \ref{thm1}. We start with two intermediate results. From now on we set $q=q_2-q_1$ on $Q$ and  assume  that $q=0$ on $\R^{1+n}\setminus Q$. Without lost of generality we assume that $0\in\Omega$. Without lost of generality we  assume that for all $\omega\in\{y\in\mathbb S^{n-1}:|y-\omega_0|\leq\epsilon\} $ we have $\partial\Omega_{-,\epsilon,\omega}\subset G'$ with $\epsilon$ introduced in the previous subsection (see 2 lines before \eqref{(5.1)}). Let us consider the light-ray transform of $q$ (see \cite{RS} and \cite{St}) given by 
\[\mathcal Rq(x,\omega)=\int_\R q(t,x+t\omega)dt,\quad x\in\R^n,\ \omega\in\mathbb S^{n-1}.\]
Using the Carleman estimates introduced in  \eqref{c1a} and the geometric optics solutions of Theorem \ref{p2}  and Theorem \ref{tt1}, we obtain the following  estimate of $\mathcal R q$.
\begin{lem}\label{l100} Assume that the conditions of Theorem \ref{thm1} are fulfilled. Then, there exists  $\lambda_2>1$,  such that for all $\lambda>\lambda_2$, $\omega\in\{y\in\mathbb S^{n-1}:|y-\omega_0|\leq\epsilon\} $, we have
\begin{equation}\label{t3b}\norm{\mathcal Rq(\cdot,\omega)}_{L^1(\R^n)}\leq C\left(\lambda^{-{\alpha\over4+2\alpha}}+e^{d\lambda}\norm{B_{q_1}-B_{q_2}}\right)\end{equation}
with $\alpha=1-{n+1\over p}$ and $d, C$ depending only on  $\Omega$, $M$, $T$,  $F'$, $G'$, $p$, $n$.\end{lem}
\begin{proof}  Let $\phi\in\mathcal C^\infty_0(\R^n)$ be such that $0\leq\phi\leq1$, supp$\phi\subset\{x\in\R^n:\ |x|\leq 1\}$, $\norm{\phi}_{L^2(\R^n)}=1$.   For $0<\delta<1$, we set
\[\chi_\delta(t,x,y)=\delta^{-n/2}\phi\left(\delta^{-1}(y-x-t\omega)\right),\quad t\in\R,\ x\in \R^n,\ y\in\R^n.\]
 Note that
\begin{equation}\label{ess1}\norm{\chi_\delta(.,.,y)}_{H^k(\R^{1+n})}\leq C\delta^{-k},\quad y\in\R^n\end{equation}
with $C$ independent of $\delta$ and $y$.
 We fix $\lambda_2=\max(C\lambda_0+1,\lambda_1)^{{ \alpha+3\over\alpha}}$ with $\lambda_0$ the constant introduced in Theorem \ref{p2}, $\lambda_1$ the constant introduced in Theorem \ref{c1} and $C$ the constant of the previous estimate.
Let $\omega\in\{y\in\mathbb S^{n-1}:|y-\omega_0|\leq\epsilon\} $ and let $\lambda>\lambda_2$ with $\lambda\geq \delta^{-\alpha-3}$,  $\omega\in\{y\in\mathbb S^{n-1}:|y-\omega_0|\leq\epsilon\} $. Then, we have  $$\lambda>\lambda_2^{{\alpha\over \alpha+3}}\lambda^{{3\over \alpha+3}}\geq \lambda_0C\delta^{-3}\geq \lambda_0\norm{\chi_\delta}_{H^3(Q)}$$ and, in view of Theorem \ref{p2}, we can introduce
\[u_1(t,x)=e^{-\lambda(t+x\cdot\omega)}\left(\chi_\delta(t,x,y)+w(t,x) \right) ,\ (t,x) \in Q,\]
where $u_1\in H^2(Q)$ satisfies $\partial_t^2u_1-\Delta u_1+q_1u_1=0$ and  $w$ satisfies \eqref{p2a}. Moreover, in view of Theorem \ref{tt1}, we consider $u_2\in H_{\Box}(Q)$ a solution of \eqref{(5.1)} with $q=q_2$ of the form 
\[u_2(t,x)=e^{\lambda(t+x\cdot\omega)}\left(\chi_\delta(t,x,y)+z(t,x) \right),\quad (t,x)\in Q\]
with $z$ satisfying \eqref{CGO1b}, such that supp$\tau_{0,1}u_{2}\subset F$  and $\tau_{0,2}u_2=0$.
 Let $w_1$ be the solution of
 \bel{eq3}
\left\{
\begin{array}{ll}
 \partial_t^2w_1-\Delta w_1 +q_1w_1=0 &\mbox{in}\ Q,
\\

{w_1}_{|t=0}=0,\ {\partial_tw_1}_{|t=0}=\tau_{0,3}u_2  &\mbox{on}\ \Omega,\\
{w_1}_{|\Sigma}=\tau_{0,1}u_2.&\ 
%\\
%v(\cdot ,1,\cdot )=e^{i\theta}v(0,\cdot,\cdot)

\end{array}
\right.
\ee
Then, $u=w_1-u_2$ solves
  \bel{eq4}
\left\{\begin{array}{ll}
 \partial_t^2-\Delta u +q_1u=(q_2-q_1)u_2 &\mbox{in}\ Q,
\\
u(0,x)=\partial_tu(0,x)=0 & \mathrm{on}\  \Omega,\\

u=0 &\mbox{on}\ \Sigma
%\\
%v(\cdot ,1,\cdot )=e^{i\theta}v(0,\cdot,\cdot)
\end{array}\right.
\ee
 and since $(q_2-q_1)u_2\in L^2(Q)$, in view of Theorem A.2 in \cite{BCY}, we deduce that $u\in \mathcal C^1([0,T];L^2(\Omega))\cap \mathcal C([0,T];H^1_0(\Omega))$ with $\partial_\nu u\in L^2(\Sigma)$. Moreover we have
\[\norm{u}_{\mathcal C^1([0,T];L^2(\Omega))}+\norm{u}_{\mathcal C([0,T];H^1_0(\Omega))} +\norm{\partial_\nu u}_{L^2(\Sigma)}\leq 2CM\norm{u_2}_{L^2(Q)}.\]
Applying the Green formula with respect to $x\in\Omega$ and integration by parts with respect to $t\in(0,T)$, we find
\begin{equation}\label{t3a}\begin{aligned}\int_{Q}qu_2u_1d xdt&=\int_{Q}(\partial_t^2-\Delta+q_1)uu_1d xdt\\
\ &=-\int_{G}\partial_\nu uu_1d\sigma(x)dt-\int_{\Sigma\setminus G}\partial_\nu uu_1d\sigma(x)dt\\
\ &\ \ \ +\int_\Omega \partial_tu(T,x)u_1(T,x)dx -\int_\Omega u(T,x)\partial_tu_1(T,x)dx .\end{aligned}\end{equation}
Combining \eqref{ess1} and \eqref{p2a}, we find
\begin{equation}\label{ess2}\norm{u_1}_{H^2(Q)}\leq C\delta^{-3} e^{c\lambda}\end{equation}
with $c=T+\textrm{Diam}(\Omega)+1$, $C$ depending on $M$, $T$, $\Omega$, $n$, $p$. 
Moreover, in view of estimate \eqref{p2a}, we have
 \[\norm{w}_{L^2(\Sigma)}\leq C\norm{w}_{L^2(0,T; H^{\frac{1}{2}}(\partial\Omega))}\leq C\norm{w}_{L^2(0,T;H^{1}(\Omega))}\leq C\norm{w}_{H^1(Q)}\leq C,\]
 where $C$ depends on $\Omega$, $T$, $n$, $p$ and $M$.
Applying this estimate, \eqref{ess1}, \eqref{ess2} and the Cauchy Schwarz inequality, we obtain
\[\begin{aligned}\abs{\int_{\Sigma\setminus G}\partial_\nu uu_1d\sigma(x)dt}&\leq\int_{{\Sigma}_{+,\epsilon,\omega}}\abs{\partial_\nu ue^{-\lambda(t+x\cdot\omega)}(\chi_\delta+w)}dt d\sigma(x) \\
 \ &\leq C\delta^{-1}\left(\int_{{\Sigma}_{+,\epsilon,\omega}}\abs{e^{-\lambda(t+x\cdot\omega)}\partial_\nu u}^2d\sigma(x)dt\right)^{\frac{1}{2}},\end{aligned}\]
\[\abs{\int_{ G}\partial_\nu uu_1d\sigma(x)dt}\leq C\delta^{-1}e^{c\lambda}\norm{\partial_\nu u}_{L^2(G)}\]
 for some $C$ depending only on  $\Omega$, $T$, $n$, $p$ and $M$. Here we use the fact that $\Sigma\setminus G\subset{\Sigma}_{+,\epsilon,\omega}$
In the same way, \eqref{p2a}  imply
\[\norm{w_{\vert t=T}}_{L^2(\Omega)}\leq C\norm{w}_{H^{1}(0,T;L^2(\Omega))}\leq C\norm{w}_{H^1(Q)}\leq C\]
and 
\[\norm{\partial_tw_{\vert t=T}}_{L^2(\Omega)}\leq C\norm{w}_{H^{2}(0,T;L^2(\Omega))}\leq C\norm{w}_{H^2(Q)}\leq  C\delta^{-3}\]
with $C$ a generic constant depending on $\Omega$, $T$, $M$, $n$, $p$.
Thus, we obtain
\[\abs{\int_\Omega \partial_tu(T,x)u_1(T,x)dx}\leq C\left(\int_\Omega \abs{e^{-\lambda(T+x\cdot\omega)}\partial_tu(T,x)}^2dx\right)^{\frac{1}{2}},\]
\[\abs{\int_\Omega u(T,x)\partial_t u_1(T,x)dx}\leq C\delta^{-3}e^{c\lambda}\left(\int_\Omega |u(T,x)|^2dx\right)^{\frac{1}{2}}.\]
In view of these estimates and \eqref{t3a}, we have
\begin{equation}\label{t3cc}\begin{array}{c}\abs{\int_{Q}qu_2u_1d xdt}^2\\
\leq C\delta^{-1}\left(\int_\Omega \abs{e^{-\lambda(T+x\cdot\omega)}\partial_tu(T,x)}^2dx+\int_{{\Sigma}_{+,\epsilon,\omega}}\abs{e^{-\lambda(t+x\cdot\omega)}\partial_\nu u}^2d\sigma(x)dt\right)\\
\ \ \ C\delta^{-6}e^{2c\lambda}\left(\norm{\partial_\nu u}_{L^2(G)}^2+\norm{u_{t=T}}_{H^1(\Omega)}^2\right) \end{array}\end{equation}
where $C$ depends on $\Omega$, $T$, $M$, $n$, $p$. On the other hand,  the Carleman estimate \eqref{c1a} and the fact that ${\partial\Omega}_{+,\epsilon,\omega}\subset {\partial\Omega}_{+,\omega}$ imply
\[\begin{array}{l}\int_{{\Sigma}_{+,\epsilon,\omega}}\abs{e^{-\lambda(t+x\cdot\omega)}\partial_\nu u}^2d\sigma(x)dt+\int_\Omega \abs{e^{-\lambda(T+x\cdot\omega)}\partial_tu(T,x)}^2dx\\
\ \\

\leq \epsilon^{-1}\left(\int_{{\Sigma}_{+,\omega}}\abs{e^{-\lambda(t+x\cdot\omega)}\partial_\nu u}^2\omega\cdot \nu(x)d\sigma(x)dt+\int_\Omega \abs{e^{-\lambda(T+x\cdot\omega)}\partial_tu(T,x)}^2dx\right)\\
\ \\
\leq {\epsilon^{-1}C\over \lambda}\left(\int_Q\abs{ e^{-\lambda(t+x\cdot\omega)}(\partial_t^2-\Delta  +q_1)u}^2dxdt+\int_{{\Sigma}_{-,\omega}}\abs{e^{-\lambda(t+x\cdot\omega)}\partial_\nu u}^2\abs{\omega\cdot \nu(x)}d\sigma(x)dt\right)\\
\ \ \ +\epsilon^{-1}C \left(\lambda^3\int_\Omega e^{-2\lambda(T+x\cdot\omega)}\abs{u_{\vert t=T}}^2dx+\lambda\int_\Omega e^{-2\lambda(T+x\cdot\omega)}\abs{\nabla_xu_{\vert t=T}}^2dx\right)\\
\ \\
\leq {\epsilon^{-1}C\over\lambda}\left(\int_Q\abs{ \abs{ q}^2(1+\abs{z})^2}^2dxdt\right)+C\delta^{-6}e^{2c\lambda} (\norm{\partial_\nu u}_{L^2(G)}^2+\norm{u_{t=T}}_{H^1(\Omega)}^2)\end{array}\]
Combining this with \eqref{t3cc}, we obtain
\begin{equation}\label{t2f}\abs{\int_{Q}qu_1u_2d xdt}^2
\leq {C\over \lambda} +C\delta^{-6}e^{2c\lambda} (\norm{\partial_\nu u}_{L^2(G)}^2+\norm{u_{t=T}}_{H^1(\Omega)}^2)\end{equation}
with $C$ depending only on  $\Omega$, $T$, $G'$, $M$, $n$, $p$. On the other hand,  we have
\[\int_{Q}qu_1u_2d xdt=\int_{\R^{1+n}}q(t,x)\chi_\delta^2(t,x,y)dxdt+ \int_{Q}Z(t,x)  dxdt\]
with $ Z=q(z\chi_\delta+w\chi_\delta+zw)$. Then, in view of \eqref{p2a}  and \eqref{CGO1b}, an application of the Cauchy-Schwarz inequality yields
\[\abs{\int_{Q}Z(t,x)  dxdt}\leq C(\delta^{-2}\lambda^{-\frac{1}{2}}+\delta^{-3}\lambda^{-1})\]
with $C$ depending  on $\Omega$, $T$, $G'$, $M$. Combining this estimate with \eqref{t2f}, we obtain
\[\abs{V_{\delta,q}(y)}^2\leq
 C \left(\delta^{-4}\lambda^{-1}+\delta^{-6}\lambda^{-2} +\delta^{-6}e^{2c\lambda} (\norm{\partial_\nu u}_{L^2(G)}^2+\norm{u_{t=T}}_{H^1(\Omega)}^2)\right)\]
with 
\[V_{\delta,q}(y)=\int_{\R^{1+n}}q(t,x)\delta^{-n}\phi^2(\delta^{-1}(y-x-t\omega))dxdt=\int_{\R^n}\left(\int_\R q(t,x+t\omega)dt\right)\delta^{-n}\phi^2(\delta^{-1}(y-x))dx,\quad y\in\R^n.\]
On the other hand, one can check that supp$V_{\delta,q}\subset\{y:\ |y|\leq T+\textrm{Diam}(\Omega)+1\}$ and from the previous estimate we get
\begin{equation}\label{t33d}\norm{V_{\delta,q}}_{L^1(\R^n)}\leq
 C \left(\delta^{-2}\lambda^{-1/2}+\delta^{-3}\lambda^{-1} +\delta^{-3}e^{c\lambda} (\norm{\partial_\nu u}_{L^2(G)}+\norm{u_{t=T}}_{H^1(\Omega)})\right)\end{equation}
In order to complete the proof we only need to check that this estimates implies \eqref{t3b}. For this purpose using the fact that $q\in W^{1,p}(Q)$ with $p>n+1$, by the Sobolev embedding theorem (e.g. Theorem 1.4.4.1 of \cite{Gr}) we have $q\in \mathcal C^{\alpha}(\overline{Q})$ with
$\norm{q}_{\mathcal C^{\alpha}(\overline{Q})}\leq CM$ with $C$ depending on $\Omega$, $n$, $p$ and $T$. Moreover, applying \eqref{thmma}, we deduce that for all $t\in\R$, $q(t,\cdot)\in\mathcal C^{\alpha}(\R^n)$. Thus, using the fact that
\[\mathcal Rq(x,\omega)=\int_0^Tq(t,x+t\omega)dt,\]
we deduce that  $\mathcal Rq(\cdot,\omega)\in \mathcal C^{\alpha}(\R^n)$ and supp$\mathcal Rq(\cdot,\omega)\subset\{x\in\R^n:\ |x|\leq \textrm{Diam}(\Omega)+T\}$. Combining this with the fact that
\[V_{\delta,q}(y)=\int_{\R^n}\mathcal Rq(y-\delta u,\omega)\phi^2(u)du,\]
we obtain
\[\norm{V_{\delta,q}-\mathcal Rq(\cdot,\omega)}_{L^1(\R^n)}\leq C\delta^{\alpha}\]
with $C$ depending on $\Omega$, $T$, $M$, $p$ and $n$.
 Combining this with  \eqref{t33d} we deduce \eqref{t3b} by using the fact that 
\[\begin{aligned}\norm{\partial_\nu u}_{L^2(G)}^2+\norm{u_{t=T}}_{H^1(\Omega)}^2&\leq \norm{B_{q_1}-B_{q_2}}^2\norm{(\tau_{0,1}u_2,\tau_{0,3}u_2)}^2_{\mathcal H_F(\partial Q)}\\
\ &\leq C\norm{B_{q_1}-B_{q_2}}^2\norm{u_2}^2_{H_\Box(Q)}\\
\ &\leq C\norm{B_{q_1}-B_{q_2}}^2\norm{u_2}^2_{L^2(Q)}\\
\ &\leq C\delta^{-4}e^{2c\lambda}\end{aligned}\]
and by choosing $\delta=\lambda^{-{1\over4+2\alpha}}$ and $d=2c+1$. Here we have used \eqref{CGO1b} and the fact that
\[\norm{u_2}^2_{H_\Box(Q)}=(1+\norm{q_2}_{L^\infty(Q)}^2)\norm{u_2}^2_{L^2(Q)}.\]
\end{proof}

From now on, for all $r>0$, we denote by $B_r$ the set $B_r= \{z\in\R^{1+n}:\ \abs{z}<r\}$.
Let us recall the following result, which follows from Theorem 3 in \cite{AE} (see also \cite{V}), on the continuous dependence in the analytic continuation problem.
\begin{prop}\label{p8}  Let $\rho>0$ and assume that $f:\ B_{2\rho}\subset \R^{1+n}\to \C$ is a real analytic function satisfying  
\[\norm{\partial^\beta f}_{L^\infty(B_{2\rho})}\leq \frac{N\beta!}{(\rho \lambda)^{\abs{\beta}}},\quad \beta\in\mathbb N^{1+n}\]
for some $N>0$ and $0<\lambda\leq1$. Further let $E\subset B_{\frac{\rho}{2}}$ be a  measurable set with strictly positive Lebesgue measure. Then,
\[\norm{ f}_{L^\infty(B_\rho)}\leq C(N)^{(1-b)}\left(\norm{ f}_{L^\infty(E)}\right)^{b},\]
where $b\in(0,1)$, $C>0$ depend  on   $\lambda$, $\abs{E}$ and $\rho$.\end{prop}
Armed with Lemma \ref{l100}, we will use Proposition \ref{p8} to complete the proof of  Theorem \ref{thm1}.
\ \\
\textbf{Proof of Theorem \ref{thm1}.}  We set $U=\{y\in\mathbb S^{n-1}:|y-\omega_0|\leq\epsilon\}$. For all $\xi\in\R^n$ we introduce 
\[a(\xi)=\inf_{\omega\in U}\xi\cdot\omega,\quad b(\xi)=\sup_{\omega\in U}\xi\cdot\omega.\]
Consider the set $E_1=\{(\tau,\xi):\ \xi\in\R^n,\ a(\xi)\leq\tau\leq b(\xi)\}$ and note that for all $(\tau,\xi)\in E_1$ there exists $\omega\in U$ such that $\tau=\xi\cdot\omega$. It is clear  that 
\[(2\pi)^{-n/2}\int_{\R^n}\mathcal Rq(x,\omega)e^{-ix\cdot\xi}dx=(2\pi)^{-n/2}\int_{\R^{n+1}}q(t,x)e^{-i(t\omega\cdot\xi+x\cdot\xi)}dtdx=(2\pi)\mathcal F(q)(\omega\cdot\xi,\xi),\ \xi\in\R^n,\]
where $\mathcal F(q)$ is the Fourier transform of $q$ given by 
\[\mathcal F(q)(\tau,\xi)=(2\pi)^{-(n+1)/2}\int_{\R^{n+1}}q(t,x)e^{-i(t\tau+x\cdot\xi)}dtdx.\]
Thus,  we get 
\[\abs{\mathcal F(q)(\tau,\xi)}\leq (2\pi)^{-(n+1)/2} \sup_{\omega\in U}\norm{\mathcal Rq(.,\omega)}_{L^1(\R^n)},\quad (\tau,\xi)\in E_1.\]
 Combining this with \eqref{t3b} we obtain 
\begin{equation}\label{ef}\abs{\mathcal F(q)(\tau,\xi)}\leq C\left(\lambda^{-{\alpha\over4+2\alpha}}+e^{d\lambda}\norm{B_{q_1}-B_{q_2}}\right),\quad (\tau,\xi)\in E_1.\end{equation}
 We set for fixed $R>0$, which will be made precise later, and $(\tau,\xi)\in\R^{1+n}$, 
\[H(\tau,\xi)=\mathcal F(q)(R(\tau,\xi))=(2\pi)^{-\frac{(1+n)}{2}}\int_{\R^{1+n}}q(t,x)e^{-iR(\tau,\xi)\cdot (t,x)}dxdt.\]
 Since supp$q\subset\overline{Q}$ and $0\in\Omega$, $H$ is real analytic and
\[\abs{\partial^\beta H(\tau,\xi)}\leq C\frac{\norm{q}_{L^1(Q)}R^{\abs{\beta}}}{([\max(T,\textrm{Diam}(\Omega))]^{-1})^{\abs{\beta}}}\leq C\norm{q}_{L^1(Q)}\frac{R^{\abs{\beta}}}{\beta!([\max(T,\textrm{Diam}(\Omega))]^{-1})^{\abs{\beta}}}\beta!,\quad\beta\in\mathbb N^{1+n}\]
 with $C$ depending on  $T$, $n$, $p$ and $\Omega$.
Moreover, we have
\[\norm{q}_{L^1(Q)}\leq 2M(T\abs{\Omega})^{1-{1\over p}}\]
and one can check that
\[\frac{R^{\abs{\beta}}}{\beta!}\leq e^{(1+n)R}.\]
Applying these estimates, we obtain
\begin{equation}\label{lele}\abs{\partial^\beta H(\tau,\xi)}\leq C\frac{e^{(1+n)R}\beta!}{([\max(T,\textrm{Diam}(\Omega))]^{-1})^{\abs{\beta}}},\quad\beta\in\mathbb N^{1+n}\end{equation}
with $C$ depending on $M$, $\Omega$ and $T$. Set  $\rho=[\max(T,\textrm{Diam}(\Omega))]^{-1}+1$,   $E=E_1\cap\{\zeta\in\R^{1+n}:\ \abs{\zeta}<\min(\frac{\rho}{2},1)\}$ with $N=Ce^{(1+n)R}$ and $\lambda=\frac{[\max(T,\textrm{Diam}(\Omega))]^{-1}}{\rho}$.   In view of \eqref{lele}, we have
\[\norm{\partial^\beta H}_{L^\infty(B_{2\rho})}\leq C\frac{e^{(1+n)R}\beta!}{([\max(T,\textrm{Diam}(\Omega))]^{-1})^{\abs{\beta}}}=\frac{N\beta!}{(\rho \lambda)^{\abs{\beta}}},\quad\beta\in\mathbb N^{1+n}.\]
Since for all $(\tau,\xi)\in E_1$ we have $|(\tau,\xi)|^2\leq 2|\xi|^2$, one can check that $$\mathcal A=\{(\tau,\xi):\ (\tau,\xi)\in E_1,\ \xi/|\xi|\in U,\ \abs{\xi}<r\}\subset E$$ with $r={\min(\rho/2,1)\over \sqrt{2}}$. On the other hand, for all $(\tau,\xi)\in\mathcal A$, we have $b(\xi)=\abs{\xi}$ and for any $\omega\in U\setminus\{\pm\xi/|\xi|\}$ we have $\omega\cdot\xi<|\xi|=b(\xi)$. Therefore, for  all $(\tau,\xi)\in\mathcal A$ we have $a(\xi)<b(\xi)$ and, by fixing $W=\{\xi\in\R^n:\ \xi/|\xi|\in U,\ |\xi|<r\}$, we get
\[|\mathcal A|=\int_{\mathcal A}d\tau d\xi=\int_{W}\int_{a(\xi)}^{b(\xi)}d\tau d\xi=\int_{W}(b(\xi)-a(\xi))d\tau d\xi>0.\]
Here we have used the fact that $a,b\in\mathcal C(\R^n)$ and the fact that $|W|>0$.
Thus, we have $\abs{E}>0$. Moreover, one can easily check that for all $(\tau,\xi)\in E_1$, $r>0$, $(r\tau,r\xi)\in E_1$.
Then, since $E\subset B_{\frac{\rho}{2}}$, $0<\lambda<1$ and $\rho>1$, applying Proposition \ref{p8} to $H$ we obtain
\[\abs{\mathcal F(q)(R(\tau,\xi))}=\abs{ H(\tau,\xi)}\leq \norm{H}_{L^\infty(B_\rho)}\leq Ce^{(1+n)R(1-b)}\left(\norm{H}_{L^\infty(E)}\right)^b,\quad \abs{(\tau,\xi)}<1,\]
where $C>0$ and $0<b<1$ depend only on $\Omega$, $T$, $M$, $F'$ and $G'$. 
But, estimate \eqref{t3b} implies that
\[\abs{H(\tau,\xi)}^2=\abs{\mathcal F(q)(R(\tau,\xi))}^2\leq C\left(\lambda^{-{2\alpha\over4+2\alpha}}+e^{2d\lambda}\norm{B_{q_1}-B_{q_2}}^2\right),\ (\tau,\xi)\in E\]
and we deduce
\begin{equation}\label{t3c}\abs{\mathcal F(q)(\tau,\xi)}^2\leq Ce^{2(1+n)(1-b)R}\left(\lambda^{-{2\alpha\over4+2\alpha}}+e^{2d\lambda}\norm{B_{q_1}-B_{q_2}}^2\right)^b,\ \abs{(\tau,\xi)}<R.\end{equation}
Note that
\begin{equation}\label{t3d}
\| q \|_{H^{-1}(\R^{1+n})}^{\frac{2}{b}}\leq C \left( \int_{\R^{1+n}} (1+\abs{(\tau,\xi)}^2)^{-1}| \mathcal F(q)(\tau,\xi)|^2  dydl\right)^{\frac{1}{b}}.
\end{equation}
 We shall make precise below,  \[\int_{B_R} (1+\abs{(\tau,\xi)}^2)^{-1}| \mathcal F(q)(\tau,\xi)|^2  dydl\] and \[\int_{\R^{1+n} \setminus B_R} (1+\abs{(\tau,\xi)}^2)^{-1}| \mathcal F(q)(\tau,\xi)|^2  dldy\] separately.
We start by examining the last integral. 
The  Parseval-Plancherel theorem and the Sobolev embedding theorem, imply
\[\begin{aligned} \int_{\R^3 \setminus B_{R}} (1+\abs{(\tau,\xi)}^2)^{-1}| \mathcal F(q)(\tau,\xi)|^2  dydl 
 &\leq  \frac{1}{R^2} \int_{\R^{1+n} \setminus B_{R}}
| \mathcal F(q)(\tau,\xi) |^2  dydl \\
\ & \leq  \frac{1}{R^2} \int_{\R^{1+n}} | \mathcal F(q)(\tau,\xi) |^2  dydl=\frac{1}{R^2} \int_{\R^{1+n}} | q(t,x) |^2 dtdx \\ 
\ & \leq  \frac{4(T\abs{\Omega})^{{p-2\over p}}M^2}{R^2}.\end{aligned}\] 
We end up getting that
\bel{t3f}
\int_{\R^{1+n} \setminus B_{R}} (1+\abs{(\tau,\xi)}^2)^{-1}| \mathcal F(q)(\tau,\xi)|^2  dydl\leq \frac{C}{R^2}.
\ee
Further, in light of \eqref{t3c}, we get
\bel{t3h}
\int_{ B_{R} }(1+\abs{(\tau,\xi)}^2)^{-1} | \mathcal F(q)(\tau,\xi)|^2  dydl
\leq CR^{1+n}e^{2(1+n)(1-b)R}\left(\lambda^{-{2\alpha\over4+2\alpha}}+e^{2d\lambda}\norm{B_{q_1}-B_{q_2}}^2\right)^b,
\ee
upon eventually substituting $C$ for some suitable algebraic expression of $C$.

Last, putting \eqref{t3f}--\eqref{t3h} together we find out that
\bel{t3j}\begin{aligned}
\| q \|_{H^{-1}(\R^{1+n})}^{\frac{2}{b}}&\leq C \left(\frac{1}{R^2}+R^{1+n}e^{2(1+n)(1-b)R}\left(\lambda^{-{2\alpha\over4+2\alpha}}+e^{2d\lambda}\norm{B_{q_1}-B_{q_2}}^2\right)^b\right)^{\frac{1}{b}}\\
\ &\leq C\left(R^{-\frac{2}{b}}+\lambda^{-{2\alpha\over4+2\alpha}}R^{\frac{n+1}{b}}e^{2(1+n)(\frac{1-b}{b})R}+R^{\frac{n+1}{b}}e^{2(1+n)(\frac{1-b}{b})R}e^{2d\lambda}\norm{B_{q_1}-B_{q_2}}^2\right),\end{aligned}
\ee
for $\lambda>\lambda_2$ where the constant $C>0$ depends only on  $\Omega$, $T$, $F'$, $G'$, $n$, $p$ and $M$. Here we have used the fact that $x\mapsto x^{\frac{1}{b}}$ is convex on $(0,+\infty)$ since $b\in(0,1)$. Now let $R_1>1$ be such that 
 \[R^{\frac{n+3}{b}}e^{2(1+n)R(\frac{1-b}{b})}>\lambda_2^{{2\alpha\over4+2\alpha}},\quad R>R_1.\]
Then, choosing $\lambda^{{2\alpha\over4+2\alpha}}=R^{\frac{n+3}{b}}e^{2(1+n)R(\frac{1-b}{b})}$ we have $\lambda>\lambda_2$ and $\lambda^{-{2\alpha\over4+2\alpha}}R^{\frac{n+1}{b}}e^{2(1+n)(\frac{1-b}{b})R}=R^{-\frac{2}{b}}$. With this value of $\lambda$ we obtain
\begin{equation}\label{t3k}\| q \|_{H^{-1}(\R^{1+n})}^{\frac{2}{b}}\leq C\left(R^{-\frac{2}{b}}+R^{\frac{n+1}{b}}e^{2(1+n)(\frac{1-b}{b})R}\textrm{exp}\left(2d{R}^{\frac{(n+3)(4+2\alpha)}{2\alpha b}}e^{(4+2\alpha)(1+n)R(\frac{1-b}{\alpha b})}\right)\norm{B_{q_1}-B_{q_2}}^2\right).\end{equation}
On the other hand, we have
\[\begin{array}{l}R^{\frac{n+1}{b}}e^{2(1+n)(\frac{1-b}{b})R}\textrm{exp}\left(2d{R}^{\frac{(n+3)(4+2\alpha)}{2\alpha b}}e^{(4+2\alpha)(1+n)R(\frac{1-b}{\alpha b})}\right)\\
\leq \textrm{exp}\left(R^{\frac{n+1}{b}}+2(1+n)\left(\frac{1-b}{b}\right)R+2d{R}^{\frac{(n+3)(4+2\alpha)}{2\alpha b}}e^{(4+2\alpha)(1+n)R(\frac{1-b}{\alpha b})}\right)\\
\leq \textrm{exp}\left(e^{\frac{n+1}{b}R}+e^{2(1+n)\left(\frac{1-b}{b}\right)R}+e^{[2d+\frac{(n+3)(4+2\alpha)}{2\alpha b}+(4+2\alpha)(1+n)(\frac{1-b}{\alpha b})]R}\right)\\
\leq \textrm{exp}\left(3e^{[\frac{n+1}{b}+2(1+n)(\frac{1-b}{b})+2d+\frac{(n+3)(4+2\alpha)}{2\alpha b}+(4+2\alpha)(1+n)(\frac{1-b}{\alpha b})]R}\right)\\
\leq \textrm{exp}\left(e^{[3+\frac{n+1}{b}+2(1+n)(\frac{1-b}{b})+2d+\frac{(n+3)(4+2\alpha)}{2\alpha b}+(4+2\alpha)(1+n)(\frac{1-b}{\alpha b})]R}\right).\end{array}\]
Setting $A=3+\frac{n+1}{b}+2(1+n)(\frac{1-b}{b})+2d+\frac{(n+3)(4+2\alpha)}{2\alpha b}+(4+2\alpha)(1+n)(\frac{1-b}{\alpha b})$, \eqref{t3k} leads to
\begin{equation}\label{t3l}\| q \|_{H^{-1}(\R^{1+n})}^{\frac{2}{b}}\leq C\left(R^{-\frac{2}{b}}+e^{{e}^{AR}}\norm{B_{q_1}-B_{q_2}}^2\right),\quad R>R_1.\end{equation}
Set $\gamma=\norm{B_{q_1}-B_{q_2}}$ and $\gamma^*=e^{-{e}^{AR_1}}$. For $\gamma\geq \gamma^*$ we have
\begin{equation}\label{t3m}\| q \|_{H^{-1}(Q)}\leq C\| q \|_{L^\infty(Q)}\leq \frac{2CM}{{\gamma^*}^2}\gamma^2.\end{equation}
For $0<\gamma<\gamma^*$, by taking
$R=R_2=\frac{1}{A} \ln(\abs{\ln \gamma})$  in \eqref{t3l}, which is permitted since $R_2> R_1$, we find out that
\[\| q \|_{H^{-1}(Q)}\leq \| q \|_{H^{-1}(\R^{1+n})}\leq C\ln(\abs{\ln \gamma})^{-1}\left(\ln(\abs{\ln \gamma})^{\frac{2}{b}}\gamma+A^{\frac{2}{b}}\right)^\frac{b}{2}.\]
 Now, since 
$\underset{0<\gamma \leq \gamma_*}{\sup} \left(\ln(\abs{\ln \gamma})^{\frac{2}{b}}\gamma+A^{\frac{2}{b}}\right)^\frac{b}{2}$ is just another constant depending only on  $\Omega$, $T$, $F'$, $G'$, $n$, $p$ and $M$, we obtain
\begin{equation}\label{t3n}\| q \|_{H^{-1}(Q)}\leq C\ln(\abs{\ln \gamma})^{-1},\quad 0<\gamma<\gamma^*.\end{equation}
 By interpolation we find
\[\| q \|_{L^2(Q)}\leq C\| q \|_{H^{1}(Q)}^{\frac{1}{2}}\| q \|_{H^{-1}(Q)}^{\frac{1}{2}}\leq C\left(4(n+2)M^2(\abs{\Omega}T)^{{p-2\over 2p}}\right)^{\frac{1}{2}}\| q \|_{H^{-1}(Q)}^{\frac{1}{2}}\]
with $C$ depending only on  $\Omega$ and $T$. Combining this  estimate with \eqref{t3n}, we deduce  \eqref{thm1a}.
\qed

\end{document}